\documentclass{amsart}%
\usepackage{amsmath}
\usepackage{color}
\usepackage{amssymb}
\usepackage{amsfonts}
\usepackage{graphicx}%

\newtheorem{theorem}{Theorem}[section]
\theoremstyle{plain}

\newtheorem{corollary}[theorem]{Corollary}

\newtheorem{lemma}[theorem]{Lemma}

\newtheorem{proposition}[theorem]{Proposition}

\numberwithin{equation}{section}
\newcommand{\re}{\mathbb{R}}

\def \R{\mathbb R}

\def \S{\mathbb S}

\def \N{\mathbb N}

\begin{document}
\title[Low energy nodal solutions to the Yamabe equation on spheres]{Low energy nodal solutions to the  Yamabe equation }
\author{Juan Carlos Fern\'{a}ndez}
\address{Centro de Investigaci\'{o}n en Matem\'{a}ticas, CIMAT, Calle Jalisco s/n, 36023 Guanajuato, Guanajuato, M\'{e}xico}
\email{juan.fernandez@cimat.mx}
\author{Jimmy Petean}
\address{Centro de Investigaci\'{o}n en Matem\'{a}ticas, CIMAT, Calle Jalisco s/n, 36023 Guanajuato, Guanajuato, M\'{e}xico}
\email{jimmy@cimat.mx}%
\thanks{The authors were supported by grant 220074 of Fondo Sectorial de Investigaci\'{o}n para la Educaci\'{o}n SEP-CONACYT. J.C. Fern\'{a}ndez was supported by a postdoctoral fellowship of CONACYT}
\date{\today}

\begin{abstract}
Given an isoparametric function $f$ on the $n$-dimensional sphere, we consider the space of functions $w\circ f$ to reduce the Yamabe equation on the round sphere into a singular ODE on $w$ in the interval $[0,\pi]$, of the form $w''+(h(r)/\sin r)w'+\lambda(\vert w\vert^{4/n-2}w-w)=0$ and boundary conditions $w'(0)=0=w'(\pi)$, where $h$ is a monotone function with exactly one zero on $(0,\pi)$ and $\lambda>0$ is a constant. For any positive integer $k$ we obtain a solution with exactly $k$-zeroes yielding solutions to the Yamabe equation with exactly $k$ connected isoparametric hypersurfaces as nodal set. The idea of the proof is to consider the initial value problems on the singularities, and then to solve the corresponding {\it double shooting} problem, matching the values of $w$ and $w'$ at the unique zero of $h$. In particular we obtain solutions with exactly one zero, providing solutions of the Yamabe equation with low energy.

\textsc{Key words: } Singular ODE; Yamabe problem; nodal solution; isoparametric hypersurfaces; shooting method.

\textsc{2010 MSC: } 34B16, 35B06, 35B33,  53C21, 58J05.

\bigskip

\end{abstract}
\maketitle

\section{\textbf{Introduction}}

The aim of this article is to provide new examples of nodal solutions to the Yamabe equation on the round sphere with a prescribed number of nodal domains.

Given a compact Riemannian manifold $(M,g)$ without boundary of dimension
$n\geq3$, the Yamabe problem consists in finding a metric $\hat{g}$
conformally equivalent to $g$ with constant scalar curvature. In \cite{y}, H. Yamabe considered the infimum  of the (normalized) total scalar functional restricted to the conformal class $[g]$, consisting of all metrics conformal to $g$,
\[
Y(M,[g]):=\inf_{h\in[g]}\frac{\int_M s_h \  dV_h}{(Vol(M,h))^{\frac{n-2}{n}}},
\]
where $dV_h$ denotes the volume element of $h$ and $s_h$ is the scalar curvature of $h$. The critical points of the total scalar curvature functional restricted to $[g]$ are
the metrics in $[g]$ which have constant scalar curvature. Yamabe attempted to prove
the existence of constant scalar curvature metrics in $[g]$ by showing that $Y(M, [g])$ is
realized. His proof contained a mistake but his statement was eventually proved to
be correct in a series of beautiful articles by N. Trudinger \cite{t}, T. Aubin \cite{a1} and R.
Schoen \cite{sch}.

Writing a conformal metric $h\in[g]$ as $h=u^{4/(n-2)}g$, where
$u\in\mathcal{C}^{\infty}(M),$ $u>0$, the problem turns out to be equivalent to solving the nonlinear PDE with critical exponent
\begin{equation} \label{yamabe}%
-\Delta_{g}u+\mathfrak{c}_{n}s_{g}u=\kappa\left\vert u\right\vert ^{p_n%
-1}u,\qquad u\in\mathcal{C}^{\infty}(M)\text{,}
\end{equation}
where $\Delta_{g}=\,$div$_{g}\nabla_{g}$ is the Laplace-Beltrami operator,
$\mathfrak{c}_{n}:=\frac{n-2}{4(n-1)}$,  $\kappa\in\mathbb{R}$, and $p_n:=\frac{n+2}{n-2}$ is the critical
Sobolev exponent. In fact, $h$ has constant scalar
curvature $\kappa$\ \ iff $u$ is a positive solution to this
problem. In these terms, the total scalar functional restricted to $[g]$ becomes the so called Yamabe functional
\[
Y(u):=\frac{\int_M \vert\nabla u\vert_g +\mathfrak{c}_{n}s_g u^2 dV_g}{\left(\int \vert u\vert^{p_n+1}dV_g\right)^{2/(p_n+1)}},
\]
and the Yamabe equation is the Euler-Lagrange equation of this functional. Solutions in general are not unique and there have been many results studying the set of positive solutions, see for instance \cite{b,bm,hv,kms,o,po}. Less is known about the existence and multiplicity of nodal solutions. If $u$ is a sign-changing solution to problem (\ref{yamabe}), then\textit{ }$h=\left\vert u\right\vert ^{4/(n-2)}g$ is not a
metric, as $h$ is not smooth and it vanishes on the set of zeroes of
$u$. B. Ammann and E. Humbert \cite{ah} called $h$ a \textit{generalized metric}. In
\cite{ah} and \cite{es}, the existence of at least one nodal solution with minimal energy was settled when the manifold is not locally conformally flat and its dimension is at least 11. To prove this, if $\lambda_i(g)$ denotes the $i$-th eigenvalue of the operator $L_g=-\Delta_{g}+\mathfrak{c}_{n}s_{g}$, the authors showed that the infimum
\[
\mu_2(M,[g]):=\inf_{h\in[g]}\lambda_2(h)Vol(M,h)^{2/n}
\]
is achieved by a generalized metric, the conformal factor being the absolute value of a nodal solution. This infimum is called \emph{the second Yamabe invariant}.

However, multiplicity of nodal solutions to the Yamabe problem
(\ref{yamabe}) is, largely, an open question. Other existence and multiplicity results have been obtained, for instance, in \cite{cfer2,hen,pet} in the case of products and in the presence of symmetries. In a classical paper \cite{di},
W.Y. Ding established the existence of infinitely many nodal solutions to this
problem on the standard sphere $\mathbb{S}^{n}.$ He took advantage of the fact
that $\mathbb{S}^{n}$ is invariant under the action of isometry groups whose
orbits are positive dimensional.

The Yamabe problem on the round sphere $(\S^n,g_0^n)$,
\begin{equation}\label{Eq:Yamabe sphere}
-\Delta_{g_0^n}u+\frac{n(n-2)}{4}u=\frac{n(n-2)}{4}\left\vert u\right\vert ^{p_n-1}u\qquad\text{ on }\mathbb{S}^n
\end{equation}
is crucial to understand the problem on other closed manifolds \cite{a1}. Unlike what happens with the positive solutions to this problem on the sphere, where Aubin \cite{a1} described them all, a classification of all the nodal solutions is far for being complete, and they present many interesting and diverse behaviours. For instance, in \cite{cfer2}, using variational methods, the authors showed that if $\Gamma$ is a compact group of isometries of the round $n$-dimensional sphere such that the every $\Gamma$-orbit has positive dimension, then the Yamabe problem \eqref{Eq:Yamabe sphere} has infinitely many sign changing $\Gamma$-invariant solutions. This result generalizes Ding's result when considering the
case  $\Gamma=O(k)\times O(m), m+k=n+1$, $m,k\geq 2$.
Other multiplicity results for this equation and small perturbations of \eqref{yamabe} were obtained using the Lyapunov-Schmidt reduction method, for instance, in \cite{dpmpp2,epv,mpv,mw,rv1}.

It is well known that the Yamabe problem on the sphere is equivalent, via the stereographic projection, to the Yamabe problem on $\R^n$,
\begin{equation}\label{Eq: Yamabe in R^n}
\left\{\begin{tabular}{cc}
$-\Delta u=\left\vert u\right\vert ^{p_n-1}u$ & in $\mathbb{R}^n$\\
$u\in D^{1,2}(\R^n),$ &
\end{tabular}\right.
\end{equation}
where $D^{1,2}(\R^n)$ denotes the completion of $\mathcal{C}_c^\infty(\R^n)$ with respect to the norm $\Vert u\Vert:=\int_{\R^n}\vert\nabla u\vert^2 dx$. Multiple nodal solutions to this problem have been constructed using the \emph{standard bubble}
\begin{equation}\label{Eq:Bubble}
U(x):=[n(n-2)]^{\frac{n-2}{4}}\frac{1}{[1+\vert x\vert^2]^{\frac{n-2}{2}}},
\end{equation}
as building block \cite{dpmpp,mw}. These solutions look like a positive bubble surrounded by $k$ negative bubbles for $k$ large enough and they differ for the ones obtained in \cite{cfer2}. Other completely different solutions were obtained very recently by M. Clapp \cite{c}, where the solutions were \emph{$\Gamma$-equivariant} for some suitable subgroups $\Gamma$ of $O(n+1)$.

\vspace{.5cm}

We will use a different approach to study problem \eqref{Eq:Yamabe sphere}, reducing it to an ODE with singular coefficients. This approach can be used under the presence of a  cohomogeneity one action or, more generally,  an \emph{isoparametric function} on the sphere. To state our main result, we briefly introduce these functions. They will be treated in more detail in Section \ref{Sec:The reduce equation}. If $(M^n,g)$ is a Riemannian manifold, then a smooth function $f:M\rightarrow[c,d]$ is an isoparametric function if there exist a continuous function $a$ and a smooth function $b$ such that
\[
\vert \nabla f\vert_g^2 = b\circ f\qquad\text{and}\qquad \Delta_gf=a\circ f.
\]
The level sets are called \emph{isoparametric hypersurfaces}. An immediate consequence of the definition of an isoparametric function is the following reduction of a PDE on $(M,g)$ into an ODE in a closed interval $[c,d]\subset\R$.
\begin{proposition}\label{Prop:Reduction}
Let $(M,g)$ be a closed Riemannian manifold and $f:M\rightarrow[c,d]$ be an isoparametric function with $c<d$, $\vert \nabla f\vert_g^2 = b\circ f$ and $\Delta_gf=a\circ f$. Then, for any function $\varphi:\R\rightarrow\R$, $v$ is a solution to the problem
\begin{equation}\label{Eq: Singular BVP}
-bv'' - av' = \varphi(v) \text{ in } [c,d],
\end{equation}
if and only if $u=v\circ f$ is a solution to the problem
\[
-\Delta_gu = \varphi(u)\quad \text{ on }\ M.
\]
\end{proposition}
The proof of this Proposition follows directly from the identity  $\Delta_g(v\circ f) = (v''\circ f)\vert \nabla f\vert_g^2 + (v'\circ f)\Delta_g f$.

In the case of space forms, E. Cartan \cite{ca1} proved that a hypersurface $M$ is isoparametric (according to the
previous definition) if and only if it has constant principal curvatures. The orbits of a cohomogeneity one action are
examples of isoparametric hypersurfaces and they are called homogeneous. The theory of isoparametric hypersurfaces
in the sphere $(\S^n,g_0^n)$ is very rich (Cf. \cite{cr} for details). If we denote by $\ell$ the number of distinct principal curvatures, H. F. M\"{u}nzner showed that $\ell=1, 2,
3, 4$ or $6$, and if $\ell$ is odd, all the multiplicities of the principal curvatures are the same, while if $\ell$ is even, there are, at most, two different multiplicities $m_1$ and $m_2$ \cite{m1,m2}. In section \ref{Sec:The reduce equation}, using the above proposition and the special properties of isoparametric functions on the round sphere \cite{cr}, we  reduce problem \eqref{Eq:Yamabe sphere} into the following singular ODE
\begin{equation}\label{Eq:Singular BVP sphere radial}
w'' + \frac{h(r)}{\sin r}w' + \frac{n(n-2)}{4 \ell^2 } [\vert w\vert^{p_n-1}w-w] = 0 \text{ on } [0,\pi],
\end{equation}
where $h(r)=\frac{m_1+m_2}{2}\cos r - \frac{m_2-m_1}{2}$. In fact, we will show that if $w$ is a solution to this equation,
with $w'(0)=w'(\pi )=0$, then $u=w(\arccos f)$ is a solution  to the Yamabe equation on the sphere \eqref{Eq:Yamabe sphere}. The case $l=1$, and then $m_1 =m_2 = n-1$, corresponds to the cohomogeneity one action of $O(n)$ on $S^n$ fixing an axis: the corresponding functions $w$ are then called {\it radial} or
{\it axially symmetric} and it is well known that, in this case,
equation \eqref{Eq:Singular BVP sphere radial} has positive solutions but no nodal solution. In all other cases
$m_1 , m_2 <n-1$.

We will prove:

\begin{theorem}\label{Th:Nodal Yamabe ODE}
If $m_1, m_2<n-1$, then equation \eqref{Eq:Singular BVP sphere radial} with boundary conditions $w'(0)=w'(\pi)=0$ admits a sequence of sign changing solutions $w_k$ having exactly $k$ zeroes in $[0,\pi]$
\end{theorem}

The function $h$ appearing in equation \eqref{Eq:Singular BVP sphere radial} has a unique zero $a_0 \in (0,\pi )$. To
prove Theorem \ref{Th:Nodal Yamabe ODE} we will consider the solutions $w_d$, $\widetilde{w}_c$ of
equation \eqref{Eq:Singular BVP sphere radial} with initial conditions $w_d'(0) = \widetilde{w}_c ' (\pi )=0$,
$w_d (0)=d$, $\widetilde{w}_c (\pi )=c$ and consider the maps $I(d) = (w_d (a_0 ), w_d ' (a_0 ))$ and
$J(c) =(\widetilde{w}_c (a_0 ), \widetilde{w}_c ' (a_0 ))$. $I, J : \re \rightarrow \re^2$ and if $I(d) = J(c)$ then
$w_d = \widetilde{w}_c$ is a solution of equation \eqref{Eq:Singular BVP sphere radial} with $w_d'(0) = w_d '(\pi )=0$.
To understand the intersections of the curves $I, J$ one needs information of the functions $w_d$, $\widetilde{w}_c$.
In Section 3 we will prove that for large $d$ and $c$ these functions have many zeroes close to $0$ and $\pi$, respectively.
This will be used in Section 4 to solve a \emph{double shooting problem} and in this way prove Theorem \ref{Th:Nodal Yamabe ODE}.

As a consequence of Theorem \ref{Th:Nodal Yamabe ODE} we obtain a multiplicity result for the Yamabe problem \eqref{Eq:Yamabe sphere}.

\begin{theorem}\label{Th:Main}
Let $S\subset \S^n$ be an isoparametric hypersurface with $\ell\neq 1$ different principal curvatures. Then, for any positive integer $k$, the Yamabe problem \eqref{Eq:Yamabe sphere} admits  a nodal solution $u_k$ such that its nodal set has exactly $k$ connected components, each of them is an isoparametric hypersurface diffeomorphic to $S$. Moreover,
\begin{equation}
\int_{\S^n}\vert u_k\vert^{p_n+1}dV_{g_0^n}\rightarrow\infty\quad\text{as}\quad k\rightarrow\infty.
\end{equation}
\end{theorem}

\begin{proof} The level sets of an isoparametric function on $\S^n$ are connected and divide $\S^n$
in two open connected components (see, for example, \cite{cr}). The existence of the solution $u_k$ is then a direct consequence
of Theorem \ref{Th:Nodal Yamabe ODE} and Proposition \ref{Prop:Reduction}.
The solution  $u_k$ has exactly $k+1$ nodal domains, which we will denote by $\Omega_i$, $i=1,\ldots,k+1$. The functions $u_{k,i}:=u_k1_{\Omega_i}\in H_{g_0^n}^1(\S^n)$, where $1_{\Omega_i}(q)=1$ if $q\in\Omega_i$ and it is zero otherwise \cite[Lemma 1]{mpf}. Define
\[
c_n:=\inf\{\int_{\S^n}\vert u\vert^{p_n+1}dV_{g_0^n}\;:\; u \text{ is a solution to }\eqref{Eq:Yamabe sphere}\}.
\]
This infimum is actually attained by the function $u\equiv 1$, so that $c_n=Vol_{g_0^n}(\S^n)>0$ \cite{a1}. A standard argument \cite[Lemma 2.4]{aw} shows that $\int_{\S^n}\vert u_{k,i}\vert^{p_n+1} dV_{g_0^n}\geq c_n$ for each $i=1,\ldots,k+1$. So
\[
\int_{\S^n}\vert u_k\vert^{p_n+1}dV_{g_0^n}=\sum_{i=1}^{k+1}\int_{\S^n}\vert u_{k,i}\vert^{p_n+1}dV_{g^n_0}\geq (k+1)c_n.
\]

\end{proof}

The classification of all the isoperimetric hypersurfaces on $\S^n$ is a very hard problem and it was posed by Cartan in 1939 \cite{ca1} and restated as Problem 34 of S. T. Yau's list of important open problems in geometry in 1990 \cite{yau}. A complete classification of all isoparametric hypersurfaces on the sphere is now available (see \cite{ch,miy2,miy3} and the references therein). For a book form presentation of this topic, we recommend \cite{cr}. Isoparametric hypersurfaces with one principal curvature in $S^n$ are the orbits of the action of $O(n)$ fixing some chosen point and its antipode, and the corresponding solutions to the Yamabe equation are the radial solutions, which are all positive \cite{a1,ta}. Writing $\R^{n+1}\equiv\R^{k}\times\R^{m}$ for $m+k=n+1,m,k\geq 2$, the isoparametric hypersurfaces with exactly two principal curvatures (i.e. $\ell =2$) are diffeomorphic to the product of spheres
\[
S_r^{k-1}\times S_s^{m-1}=\{(x,y)\in\R^{k}\times\R^{m}\;:\;\vert x\vert = r^2, \vert y\vert^2=s^2, r^2+s^2=1\}\subset\S^n,
\]
and they are the orbits of the isometric actions of $O(k)\times O(m)$ on $\S^n$. So, our result can be seen also as a generalization of Ding's result in case $\ell=2$ \cite{di}.

Next, note that for a solution $u$ of \eqref{Eq:Yamabe sphere} we have
\begin{equation}\label{Eq:Yamabe functional solutions}
Y(u)=\frac{n(n-2)}{4}\left( \int_{\S^n} \vert u\vert^{p_n+1}dV_{g^n_0}  \right)^{\frac{2}{n}}
\end{equation}
If $u$ is a solution to \eqref{Eq:Yamabe sphere}, we will refer to the quantity $\mathcal{E}(u):=\int_{\S^n}\vert u\vert^{p_n+1}dV_{g_0^n}$ as the \emph{energy of $u$}. Then, our result states that $\mathcal{E}(u_k)\rightarrow\infty$ as $k\rightarrow\infty$, the same as in Ding's result. As it was already mentioned in the proof of Theorem \ref{Th:Main}, the infimum of these energies, which we denoted by $c_n$, is attained by the constant function $u\equiv 1$. Now we address the question whether a least energy nodal solution exists or not. If the second Yamabe invariant
of a Riemannian manifold $(M,g)$ is
realized, then in \cite{ah} Ammann and Humbert proved that there is a nodal solution with least energy. But this is not
the case of the round sphere, for this infimum is never attained by a generalized metric \cite[Proposition 5.3]{ah}. As it was shown in the proof of Theorem \ref{Th:Main}, if a solution $u$ of \eqref{Eq:Yamabe sphere} changes sign, then $\mathcal{E}(u)\geq 2c_n$. However, T. Weth proved in \cite{weth} that $2c_n$ can not be the infimum of the energies, for he showed the existence of $\varepsilon>0$ such that $\mathcal{E}(u)>\varepsilon+2c_n$ for every sign-changing solution to the problem \eqref{Eq:Yamabe sphere}. It is not clear whether the $\varepsilon$ in Weth's result is large or not. For example, for the construction given by del Pino et al. \cite{dpmpp}, when the solution is constructed as a sum of one positive bubble surrounded by $k$ negative bubbles, the energy is $\mathcal{E}(u)=(k+1)c_n + O(1)$, where $O(1)$ remains bounded as $k\rightarrow\infty$. So, the energy of these solutions is, at least, $k+1$ times $c_n$, but this estimate is available only for $k$ large enough and nothing is said about energies with $k$ small. Indeed, a solution with only two bubbles, one positive and one negative, can not exist, as it was shown by F. Robert and J. V\'{e}tois in \cite{rv1}. Observe this is compatible with Weth's and Ammann-Humbert's results. Some low energy sign changing solutions were obtained by Clapp in \cite{c}. Here, the solutions have the least energy among all $\Gamma$-equivariant functions, where $\Gamma$ is a suitable subgroup of isometries of $O(n+1)$, but no energy estimates are available. Theorem \ref{Th:Main} provides several examples of low energy nodal solutions, namely, the solutions having exactly one isoparametric hypersuface as nodal set. It would be very interesting to see if one of this solutions is actually the least energy sign changing solution to the Yamabe problem on the sphere, and to compare this number with $c_n$.

Our solutions to problem the \eqref{Eq:Yamabe sphere} have the advantage that they can be computed numerically, and so their energies. We present some numerical computations for the case $\ell=2$, that is, when the isoparametric hypersufaces are the orbits of the action of $O(m)\times O(k)$ with $k+m=n+1$ and $k,m\geq 2$. In the following list we present the approximated values of the energy of the positive solution and we compare it with the energy of the solutions to some examples computed using the software $Mathematica$, for $n\leq 7$.
\begin{center}
\begin{tabular}{|c|c|c|c|c|c|}\hline
\emph{$n$} & \emph{$k$} & \emph{$m$} & \emph{$c_n$} & \emph{$\mathcal{E}$} & \emph{$\mathcal{E}$}/\emph{$c_n$}$\thickapprox$\\
\hline
3 & 2 & 2 & 19.7 & 326 & 16\\
\hline
4 & 2 & 3 & 26.3 & 362 & 13\\
\hline
5 & 2 & 4 & 31 & 370 & 12\\
\hline
5 & 3 & 3 & 31 & 509 & 16\\
\hline
6 & 2 & 5 & 33 & 350 & 10\\
\hline
6 & 3 & 4 & 33 & 535 & 16\\
\hline
7 & 2 & 6 & 32.4 & 320 & 10\\
\hline
7 & 3 & 5 & 32.4 & 492 & 15\\
\hline
7 & 4 & 4 & 32.4 & 566 & 17\\
\hline
\end{tabular}
\end{center}

This paper is organized as follows. In Section \ref{Sec:The reduce equation} we study in more detail the reduction of the PDE \eqref{Eq:Yamabe sphere} into the singular ODE \eqref{Eq:Singular BVP sphere radial} and we discuss how
to compute the energy of the solutions. In Section 3, following the argument given in \cite{McLeod}, we state and prove a theorem that guarantees the existence of solutions to problem \eqref{Eq:Singular BVP sphere radial} with arbitrarily large number of zeroes in any interval of the form $[0,A]$ and $[B,\pi]$, where $A\in(0,a_0)$ and $B\in(a_0,\pi)$. In Section \ref{Sec:Proof main}, using this result and performing a double shooting method, we prove Theorem \ref{Th:Nodal Yamabe ODE}.

\section{The reduced equation and the energy of solutions}\label{Sec:The reduce equation}

Let $f:(\S^n,g_0^n)\rightarrow\R$ be an isoparametric function. We can assume that $f$ is actually a
Cartan-Munzner polynomial (Cf. \cite{cr,hp} for details). In this situation the image of f is $[-1,1]$. As in the introduction,
we let $\ell $ be the number of distinct principal curvatures of the level sets of $f$ and let $m_1$ and $m_2$ be the two (possibly equal) multiplicities of the principal curvatures. Note that $\frac{n-1}{\ell}=\frac{m_1+m_2}{2}$ when $\ell$ is even. Then, one obtains (see \cite{hp} for details)
\[
\vert\nabla f\vert^2_{g_0^{n}}=-\ell^2f^2+\ell^2\qquad\text{and}\qquad\Delta_{g_0^n}f=-\ell(n+\ell-1)f+\frac{\ell^2(m_2-m_1)}{2}
\]
and using Proposition \ref{Prop:Reduction}, we reduce equation \eqref{Eq:Yamabe sphere} into the following ODE
\begin{equation}\label{Eq:Singular BVP sphere}
b(t)v'' + a(t)v' + \frac{n(n-2)}{4} [\vert v\vert^{p_n-1}v-v] = 0 \text{ on } [-1,1],
\end{equation}
where $a(t):=-\ell(n+\ell-1)t+\frac{\ell^2(m_2-m_1)}{2}$ and $b(t):=-(\ell^2t^2-\ell^2)$.
After the change of variables $w(r):=v(\cos r)$, problem \eqref{Eq:Singular BVP sphere} is equivalent to solving problem \eqref{Eq:Singular BVP sphere radial}. Therefore, a solution $w$ to problem \eqref{Eq:Singular BVP sphere radial} induces a solution $u:=w(\arccos f)$ to problem \eqref{Eq:Yamabe sphere}.

Observe this equation becomes singular at $r=0$ and $r=\pi$, and that the natural boundary conditions in order to obtain a smooth solution on $\S^n$ are $w'(0)=w'(\pi)=0$. Also notice that the function $h$ satisfies $h(0)=m_1$, $h(\pi ) = -m_2$, is strictly decreasing, has a unique zero $a_0\in(0,\pi)$ and $h(r)>0$ in $[0,a_0)$, while $h(r)<0$ in $(a_0,\pi]$. Moreover, the function $\widetilde{h}(r):=-h(\pi-r)=\frac{m_1 + m_2}{2} \cos r +\frac{m_2 - m_1}{2}$ has the same properties with $m_1$ and $m_2$ interchanged and a unique zero at $\pi-a_0$. To handle both singularities in \eqref{Eq:Singular BVP sphere radial} at the same time, the strategy is to shoot solutions from each singularity and expect that, for some suitable initial and final conditions, the solutions coincide. That is, we consider the initial value problem
\begin{equation}\label{Eq:Singular forward}
\left\{\begin{tabular}{cc}
$w_{i}''(r) + \frac{h(r)}{\sin r} w_i'(r) + \frac{n(n-2)}{4 \ell^2 }( |w_i(r)|^{p_n-1} w_i -w_i) =0$& in $[0,a_0]$,\\
$w_i(0)=d,\  w_i'(0)=0,$ &
\end{tabular}\right.
\end{equation}
and the ``final'' value problem
\begin{equation}\label{Eq:Singular backward}
\left\{\begin{tabular}{cc}
$w_f''(r) + \frac{h(r)}{\sin r} w_f'(r) + \frac{n(n-2)}{4 \ell^2}( |w_f(r)|^{p_n-1} w_f -w_f) =0$& in $[a_0,\pi]$,\\
$w_f(\pi)=c,\  w_f'(\pi)=0,$ &
\end{tabular}\right.
\end{equation}
and we look for initial and final conditions $d$ and $c$ such that $w_i(a_0,d)=w_f(a_0,c)$ and $w'_i(a_0,d)=w'_f(a_0,c)$, so that, by uniqueness of the solution, we have a well defined solution to problem \eqref{Eq:Singular BVP sphere radial} given by $w(r)=w_i(r,d)$ if $r\in[0,a_0]$ and $w(r)=w_f(r,c)$ if $r\in[a_0,\pi]$. To construct the solutions with an arbitrarily large number of zeroes, we will need to use that the number of zeroes before and after $a_0$ grows as $\vert d\vert,\vert c\vert\rightarrow\infty$. We will prove in Section 3 that this is the case.

Actually, problem \eqref{Eq:Singular backward} can be written as an initial condition problem having the form  \eqref{Eq:Singular forward}. Indeed, if we consider the function $\widetilde{h}(r)=-h(\pi-r)=\frac{m_1+m_2}{2}\cos r + \frac{m_2-m_1}{2}$, then $w_f$ solves \eqref{Eq:Singular backward} if and only if $\omega(r)=w_f(\pi-r)$ solves the initial value problem
\begin{equation}\label{Eq:Singular backward equivalent}
\left\{\begin{tabular}{cc}
$\omega''(r) + \frac{\widetilde{h}(r)}{\sin r} \omega'(r) + \frac{n(n-2)}{4 \ell^2}( |\omega(r)|^{p_n-1} \omega -\omega) =0$& in $[0,\pi-a_0]$,\\
$\omega(0)=c,\  \omega'(0)=0,$ &
\end{tabular}\right.
\end{equation}
So understanding equation \ref{Eq:Singular forward} is enough to also understand equation \ref{Eq:Singular backward}.

\vspace{1cm}

Now we discuss how to compute the energy of the solutions when $\ell=2$. In this case, the function $f:\S^n\rightarrow[-1,1]$ is given by $f(x,y)=\vert x\vert^2 - \vert y\vert^2$ and
the level set are  the product spheres
\[
\S^{m-1}_{\cos t}\times\S^{k-1}_{\sin t}:=\{(x,y)\in\R^m\times\R^k\;:\; \vert x\vert^2=\cos^2 t, \vert y\vert^2=\sin^2 t \},\quad t\in(0,\frac{\pi}{2}),
\]
(see, for example \cite[Chapter 3]{cr}). The corresponding multiplicities are $m_1=m-1$ and $m_2=k-1$ and the focal submanifolds are $M_0:=\S^{m-1}\times\{0\}$ and $M_{\pi/2}:=\{0\}\times \S^{k-1}$. Denote by $g_{0}^{k-1}$ and by $g_{0}^{m-1}$ the canonical metrics of $\S^{k-1}$ and $\S^{m-1}$ respectively. An explicit formula for the energy is given in the following proposition.

\begin{proposition}\label{Prop:Energy}
If $f:\S^n\rightarrow [-1,1]$ is an isoparametric function for $\ell=2$ as above, $w$ is a solution to the equation \eqref{Eq:Singular BVP sphere radial} and $u=w(\arccos f)$ is the corresponding solution to problem \eqref{Eq:Yamabe sphere}, then
\begin{equation}\label{Eq:Integral}
\int_{\S^n}\vert u\vert^{p_n+1}dV_{g_0^n}=\frac{1}{2}Vol(\S^{k-1})Vol(\S^{m-1})\int_{0}^{\pi}\vert w(r)\vert^{p_n+1}\sin^{k-1} \frac{r}{2}\cos^{m-1}\frac{r}{2}\; dr.
\end{equation}
\end{proposition}

Proposition \ref{Prop:Energy} will follow from the following more general result.

\begin{lemma}
Let $\psi:[-1,1]\rightarrow\R$ be continuous, then
\[
\int_{\S^n}\psi\circ f dV_{g_0^n}=\frac{1}{2}Vol(\S^{m-1})Vol(\S^{k-1})\int_0^\pi\psi(\cos t)\cos^{m-1}\frac{t}{2}\sin^{k-1}\frac{t}{2}\;dt.
\]
\end{lemma}

\begin{proof}
Observe we have a diffeomorphism $\varphi^{-1}:(0,\pi/2)\times\S^{m-1}\times\S^{k-1}\rightarrow\S^n\smallsetminus(M_0\cup M_{\pi/2})$ given by $\varphi^{-1}(t,x,y):=(x\cos t ,y\sin t )$. Then, we can write $g_0^n=dt^2+(\cos^2t)g_0^{m-1}+(\sin^2 t) g_0^{k-1}$  and $dV_{g_0^n}=\cos^{m-1}t\sin^{k-1}t dt\wedge dV_{g_0^{m-1}}\wedge dV_{g_0^{k-1}}$. Also observe that
\[
f\circ\varphi^{-1}(t,x,y)=\vert x \cos t \vert^2 - \vert y \sin t \vert^2 = \cos^2 t - \sin^2 t = \cos 2t.
\]
As $M_0\cup M_{\pi/2}$ has Lebesgue measure zero in $\S^n$, we have that
\begin{align*}
&\int_{\S^n}\psi\circ f dV_{g_0^n}=\int_{\S^n\smallsetminus(M_0\cup M_{\pi/2})}\psi\circ f dV_{g_0^n}\\
&=\int_{(0,\pi/2)\times\S^{m-1}\times\S^{k-1}}\psi\circ(f\circ\varphi^{-1})\cos^{m-1}t\sin^{k-1}t dt\wedge dV_{g_0^{m-1}}\wedge dV_{g_0^{k-1}}\\
&=\int_{\S^{m-1}}\left(\int_{\S^{k-1}}\left(\int_{0}^{\pi/2}\psi(\cos 2t)\cos^{m-1}t\sin^{k-1}t dt\right)dV_{g_0^{k-1}}\right)dV_{g_0^{m-1}}\\
&=Vol(\S^{m-1})Vol(\S^{k-1})\int_0^{\pi/2}\psi(\cos 2t)\cos^{m-1}t\sin^{k-1}t\; dt\\
&=\frac{1}{2}Vol(\S^{m-1})Vol(\S^{k-1})\int_0^{\pi}\psi(\cos r)\cos^{m-1}\frac{r}{2}\sin^{k-1}\frac{r}{2}\; dr,
\end{align*}
as we wanted.
\end{proof}

\begin{proof}[Proof of Proposition \ref{Prop:Energy}] Let $w$ be a solution to problem \eqref{Eq:Singular BVP sphere radial}. Then, $v:[-1,1]\rightarrow\R$ such that $w(r)=v(\cos r)$, is a solution to problem \eqref{Eq:Singular BVP sphere} and $u:=v\circ f$ is a solution to equation \eqref{Eq:Yamabe sphere}. If we define $\psi(t):=\vert v\vert^{p_n+1}$, then $\vert u\vert^{p_n}=\psi\circ f$ and the lemma yields
\begin{align*}
&\int_{\S^n}\vert u\vert^{p_n+1}dV_{g_0^n} \\
&= \frac{1}{2}Vol(\S^{m-1})Vol(\S^{k-1})\int_0^{\pi}\vert v(\cos r)\vert^{p_n+1}\cos^{m-1}\frac{r}{2}\sin^{k-1}\frac{r}{2}\: dr\\
&=\frac{1}{2}Vol(\S^{m-1})Vol(\S^{k-1})\int_0^{\pi}\vert w(r)\vert^{p_n+1}\cos^{m-1}\frac{r}{2}\sin^{k-1}\frac{r}{2}\; dr.
\end{align*}
\end{proof}

\section{Zeroes close to one of the singularities}

We can fit the initial value problem for equation \eqref{Eq:Singular BVP sphere radial} in a more general setting as follows: for constants $A>0$, $p>1$ and $\lambda>0$ and a positive $C^1$ function $H$ defined in the interval $[0,A]$, consider the following general initial condition problem

\begin{equation}\label{Eq:General singular}
\left\{\begin{tabular}{cc}
$w''(r) + \frac{H(r)}{r} w'(r) + \lambda( |w(r)|^{p-1} w -w) =0$ & in $[0,A]$\\
$w(0)=d,\  w'(0)=0.$ &
\end{tabular}\right.
\end{equation}

Equations \eqref{Eq:Singular BVP sphere radial} and \eqref{Eq:Singular backward equivalent} are special cases of the former by taking $\lambda=\frac{n(n-2)}{4 \ell^2}$, $p=p_n$, $H(r)=\frac{h(r)r}{\sin r}$ in $[0,A]$ with $A<a_0$ and $H(r)=\frac{\widetilde{h}(r)r}{\sin r}$ in $[0,A]$ with $A<\pi-a_0$. Observe, in this case, we are just dealing with the singularity at $r=0$.

As the function $H$ is positive and attains its maximum and its minimum in $[0,A]$, a standard contraction map argument yields to the local existence and uniqueness of the solutions to equation \eqref{Eq:General singular} with initial conditions $w(0)=d\in\R$ and $w'(0)=0$, depending continuously on $d$, see, for example \cite{Haraux,k}. For $d> 0$, let $w_d$ be the local solution with initial values $w_{d } (0) =d$ and  $w_{d } ' (0)=0$. To prove global existence, we introduce 
 \[
E(r):=\frac{(w_d'(r))^2}{2}+G(w_d(r)),
\]
with $G(t):=\lambda\left(\frac{\vert t\vert^{p_n+1}}{p_n+1}-\frac{t^2}{2}\right)$, which may be considered the energy function of the solution $w_d$. As $H$ is positive, then
\[
E'(r)=-\frac{H(r)}{r}(w'_d(r))^2\leq 0,
\]
for every $r\in[0,A]$. Thus, since $w_d(r)$ and $w'(r)$ can never blow up in $[0,A]$, the solution $w_d(r)$ exists in the whole interval.

This section is devoted to the proof of the following theorem:

\begin{theorem}\label{Prop:prescribed zeroes}
Suppose $H(0) >0$, $p>1$ and
\begin{equation}\label{Eq: Subcriticality ODE}
\frac{H(0) +1 }{2} < \frac{p+1}{p-1}.
\end{equation}
Then, for any $0<\varepsilon<A$ and  any positive integer $k$, there exists $d_k > 0$ so that the solution $w_{d }$ of \eqref{Eq:General singular} has at least $k$ zeroes in $(0,\varepsilon )$ for any $d\geq d_k$
\end{theorem}

The inequality \eqref{Eq: Subcriticality ODE} is true, in particular, when $n\geq 3$, $p=p_n=\frac{n+2}{n-2}$ and $H(0)<n-1$.
Also for any $p>1$ in case $H(0) \leq 1$.

The strategy of the proof is to compare the solutions $w_d$, with $d$ big enough, with the solution of a limit problem not depending on $d$, as it is done in \cite{McLeod}. Let
\[
z_{d } (r): = d ^{-\frac{2}{p-1}} \   w_{d ^{\frac{2}{p-1}}}\left( \frac{r}{d   \sqrt{\lambda} } \right) .
\]

Note that $z_{d } (0)=1$,
$z_{d } ' (0)=0$ and that this function satisfies the equation

\begin{equation}\label{Eq:General singular dilation}
z_{d }''(r) + \frac{H(\frac{r}{\sqrt{\lambda} d  })}{r} z_{d }'(r) +  |z_{d }(r)|^{p-1} z_{d } -  d ^{-2}z_{d } =0\quad{\text{in }}\ [0,dA\sqrt{\lambda}]
\end{equation}

Consider the following limit Cauchy problem

\begin{equation}\label{Eq:Limit problem}
\left\{\begin{tabular}{cc}$v''(r) + \frac{H(0)}{r} v'(r) +  |v(r)|^{p-1} v  =0$ & \text{ in } $[0,\infty)$\\
$v(0)=1$,\  $v'( 0)=0$ & \end{tabular}\right.
\end{equation}

Let $v_0$ be the unique solution to this problem. We next show that the dilated solutions $z_d$ look like $v_0$ when $d$ is big enough.

\begin{lemma}\label{Lemma:Convergence to limit problem}
For any $K>0$  the functions $z_{d}$ converge to $v_0$ $C^1$-uniformly on $[0,K]$ as $d\rightarrow\infty$.
\end{lemma}

\begin{proof} We begin by picking $D_0 >2$ such that $A \sqrt{\lambda}  D_0 >K$ and we will
assume that $d  \geq D_0$ from
now on. We then consider  $z_{d }$  defined on $[0,K]$. We divide the proof into four steps:

\textbf{Step 1.} \emph{The functions $z_d$ and $z'_d$ are uniformly bounded in $[0,K]$ for every $d\geq D_0$}

Let $G(t) = \vert t\vert^{p+1} /(p+1) -t^2 /(2d ^2 )$ and define on this interval the energy function
\[
E(r,d) = \frac{(z_{d }'(r))^2 }{2}+ G(  z_{d }(r)  ) .
\]

Then  we have
\[
E'(r,d)  = - (z_{d } ' )^2  \  \frac{H(\frac{r}{\sqrt{\lambda} d  })}{r} \leq 0
\]
for every $d\geq D_0$ and every $r\in[0,K]$. This implies that $E(r,d) $ is  bounded from above by $E(0,d)=G(1) =1/(p+1) -d ^{-2} /2$, which is bounded. It follows that $E(r,d)$ is uniformly bounded from above in both variables and, therefore,
$z_{d }$ and $z_{d }'$ are uniformly bounded in $[0,K]$ for every $d\geq D_0$.

\textbf{Step 2.} \emph{$z_{d }''$ is uniformly bounded in $[0,K]$.}

For any $\varepsilon >0$ it is clear from equation \eqref{Eq:General singular dilation} and
Step 1 that $z_{d }''$ is uniformly bounded on
$[\varepsilon , K]$. Therefore we only need to find $\varepsilon >0$ independent of $d \geq D_0$
such that $z_{d }''$ is uniformly bounded in $[0,\varepsilon ]$.

Note that $z_d''(0) <0$
for $d\geq D_0$. Therefore, close to $r=0$, both $z_{d }'$ and $z_{d }''$ are negative. Note that while $z_{d }' \leq 0$
and $z_d \in [0,1]$,
$z_{d }''$ is bounded from below by -1. Hence, if $\delta \in (0,1)$ and $r_0 >0$ is such that $z_d (r_0 ) = 1-\delta$
and $z_d' \leq 0$ in $[0,r_0 ]$ then $r_0 \geq \delta$. Note that $z_d$ has a local maximum at 0. If $r_1 >0$ is a
first local minimum and $z_d (r_1) >0$, we must have that $z_{d }^p (r_1)  -  d^{-2} z_{d } (r_1) <0$. Since
$d\geq 2$ we can find $\varepsilon >0$ independent of $d$ such that, for all $r\in [0,\varepsilon ]$, we have
that $z_d' (r) \leq 0$ and $z_d (r) \in [0,1]$. Then $z_d '' \geq -1 $ in $[0,\varepsilon]$. We can also assume that
if $r\in [0,\varepsilon ]$ then $H(\frac{r}{\sqrt{\lambda} d }) \in [H(0)/2 , 2H(0)]$. If at some $r_2 \in (0,\varepsilon )$
we have that $z_d '' (r_2 )=0$ then,

$$\frac{z_d ' (r_2 )}{r_2 } \geq \frac{-1}{H(\frac{r_2}{\sqrt{\lambda}d} )}.$$

Moreover while $z_d '' \geq 0$, the function $\frac{z_d ' (r)}{r}$ is increasing. This clearly implies that $z_d ''$ is also
bounded from above in $[0,\varepsilon ]$. This finishes the proof of Step 2.

\textbf{Step 3.} \emph{There is a sequence $d_k\rightarrow\infty$ and a function $\phi\in C^1([0,K]) \cap C^2((0,K])$ such that $z_{d_k}$ converges to $\phi$ $C^1$-uniformly on $[0,K]$, $C^2$-uniformly on $[\epsilon,K]$ for every $\epsilon>0$ and $\phi$ satisfies equation \eqref{Eq:Limit problem} in $(0,K]$}.

Taking the derivative of \eqref{Eq:General singular dilation} with respect to $r$ we get
\begin{equation}\label{Eq:formula z'''}
z'''(r) + \frac{H' (\frac{r}{\sqrt{\lambda} d  })}{ \sqrt{\lambda} d  r} z'(r) - \frac{H(\frac{r}{\sqrt{\lambda} d  })}{r^2} z'(r)  + \frac{H(\frac{r}{\sqrt{\lambda} d  })}{r} z''(r)+  p |z(r)|^{p-1} z' -  d ^{-2}z' =0
\end{equation}

The formula for $z'''$ together with Step 2 show that $z_{d }'''$ is uniformly bounded in any interval of the form
$[\epsilon , K]$, for any $\epsilon >0$. Then by Arzela-Ascoli Theorem, one can find a sequence $d _k \rightarrow \infty$ and functions $\phi$, $\alpha$ and $\beta$ such that $z_{d_k } \rightarrow \phi$, $z_{d_k }' \rightarrow \alpha$ uniformly on $[0,K]$, and $z_{d_k }'' \rightarrow \beta$ uniformly on $[\epsilon ,K]$.
By the fundamental theorem of calculus, $\alpha =\phi'$, $\beta = \phi''$, $\phi$ is a $C^1$-function on $[0,K]$ and it is
$C^2$ on $(0,K]$. It follows that $\phi$ satisfies \eqref{Eq:Limit problem} on $(0,K]$.

\textbf{Step 4.} \emph{$\phi$ is of class $C^2$ and satisfies \eqref{Eq:Limit problem} with initial conditions $\phi(0)=1$ and $\phi'(0)=0$, i.e., $\phi=v_0$ is the unique solution to the Cauchy problem \eqref{Eq:Limit problem}.}

It is enough to show that
\[
\lim_{r\rightarrow 0} \frac{\phi '(r)}{r} = -\frac{1}{1+H(0)}.
\]
for this would imply that $\phi$ is a $C^2$-function on $[0,K]$ satisfying equation \eqref{Eq:Limit problem} with the initial
conditions $\phi (0) = 1$, $\phi '(0)=0$.

We will prove that $\liminf_{r\rightarrow 0} \frac{\phi '(r)}{r} \geq -\frac{1}{1+H(0)}$. The proof that
$\limsup_{r\rightarrow 0} \frac{\phi '(r)}{r} \leq -\frac{1}{1+H(0)}$ is similar.

Then assume that $\liminf_{r\rightarrow 0} \frac{\phi '(r)}{r} <  -\frac{1}{1+H(0)}.$
Let $\delta >0$ so that, for any $r_0 >0$ there exists $r_1 \in (0,r_0 )$ with
\[
 \frac{\phi '(r_1)}{r_1} < -\frac{1}{1+H(0)}  - \delta.
\]

By taking $r_0$ small enough this would yield from equation \eqref{Eq:Limit problem} that, for any
$r\in (0,r_0 )$  satisfying the previous inequality,
\[
\phi ''(r) > -\frac{1}{1+H(0)}   >\frac{\phi '(r)}{r}.
\]

But since
\[
 \left( \frac{\phi '(r)}{r} \right) ' = \frac{1}{r^2} (\phi '' (r) r -\phi ' (r)),
\]
this in turn would imply that  $(\frac{\phi '(r)}{r})'>0$. It follows that $\frac{\phi '(r)}{r}$ is increasing in $(0,r_1 ]$ and, therefore it must have a limit
$\lim_{r\rightarrow 0} \frac{\phi '(r)}{r} =: c\leq-\frac{1}{1+H(0)}  - \delta.$ Then
$\phi ''(r)$ must also have a limit, $\widetilde{c}\geq-\frac{1}{1+H(0)} $.  But this contradicts the mean value theorem:
for any $\epsilon >0$ there must exist $r_{\epsilon} \in (\epsilon ,r)$ such that
\[
\frac{\phi '(r) -\phi '(\epsilon )}{r-\epsilon} = \phi ''(r_{\epsilon} ).
\]
But $ \phi ''(r_{\epsilon} )  > -\frac{1}{1+H(0)}  $ while
\[
\lim_{\epsilon \rightarrow 0} \frac{\phi '(r) -\phi '(\epsilon )}{r-\epsilon} =  \frac{\phi '(r)}{r} < -\frac{1}{1+H(0)}  - \delta.
\]

This proves Step 4 and, together with the previous three steps, finishes the proof of the lemma.
\end{proof}

The following Theorem is shown in \cite{Haraux}.

\begin{theorem}\label{Lemma:Infinite zeroes limit equation}
If $H(0) \geq 0$, $p>1$ and
\begin{equation}\label{Eq: Subcriticality ODE I}
\frac{H(0)+1}{2}< \frac{p+1}{p-1}
\end{equation}
then $v_0$ has infinite zeroes in $(0, \infty )$.
\end{theorem}

We do a brief remark about the proof of this Theorem. The proof is, essentially, the same as the proof of Proposition 3.9 in \cite{Haraux}. However, in this proof it is only explicitly said that $v_0$ has one zero when $v_0'(0)=0$. It is easy to see from the uniqueness of the solution and from equation \eqref{Eq:Limit problem}, that if $r$ is a zero of $v_0$, then necessarily $v'_0(r)\neq 0$ and there must exist $s>r$ such that $v'_0(s)=0$ and $v_0(s)\neq 0$. Then, the argument in \cite{Haraux} can be repeated starting at $s$ to prove the existence of another zero and so on. This is well known, it has been pointed out explicitly, for instance, in
\cite{McLeod}.

\smallskip

\begin{proof}[Proof of Theorem \ref{Prop:prescribed zeroes}]
As inequality \eqref{Eq: Subcriticality ODE I} holds true by hypothesis, Theorem \ref{Lemma:Infinite zeroes limit equation} guaranties that the solution $v_0$ to the limit problem \eqref{Eq:Limit problem} has infinite zeroes in $(0,\infty)$. Now, the number of zeroes of $w_{d^{\frac{2}{p-1}} }$ in $[0,\varepsilon )$ is the same as the number of zeroes
of $z_{d }$ in $[0,\varepsilon d  \sqrt{\lambda} )$. We pick $K$ such that $v_0$ has at least $k$ zeroes in
$[0,K]$.  Lemma \ref{Lemma:Convergence to limit problem} gives us a value of $d_k $ such that $z_{d }$ has at least $k$ zeroes in
$[0,K]$ for any $d\geq d_k$ and we can also pick $d_k $ large enough so that $K<\varepsilon d  \sqrt{\lambda}$.
\end{proof}


\section{Prescribed number of zeroes: proof of the main Theorem}\label{Sec:Proof main}

In this section we consider equation \eqref{Eq:Singular BVP sphere radial} and prove Theorem \ref{Th:Nodal Yamabe ODE}. We will restrict ourselves to the hypothesis $m_1,m_2<n-1$, so that if $w_d$ is a solution to problem \eqref{Eq:Singular BVP sphere radial}, then $u=w(\arccos f)$ could not be a radial solution of the Yamabe equation \eqref{Eq:Yamabe sphere}. The general idea of the proof was sketched in Section \ref{Sec:The reduce equation}.

Let $a_0$ be the unique zero of $h$ in $(0,\pi )$. Let $w_d$ be the solution of \eqref{Eq:Singular BVP sphere radial}
with initial conditions $w_d (0)=d$, $w_d' (0)=0$. Note that $w_{-1}, w_0$ and $w_1$ are constant functions and
that $w_{-d}=-w_d$. Note also that if $d\neq -1 , 0 , 1$ and $r$ is a critical point of $w_d$ then $w_d (r) \neq -1 , 0 ,1$;
moreover $r$ is a local minimum iff $w_d (r) \in (-\infty , -1 ) \cup (0,1)$ and a local maximum iff
$w_d (r) \in (-1,0) \cup (1,\infty )$.

Define the energy function
\[
E(r,d):=\frac{(w_d'(r))^2}{2}+G(w_d(r)),
\]
where $G(t):=\frac{n(n-2)}{4 \ell^2}\left(\frac{\vert t\vert^{p_n+1}}{p_n+1}-\frac{t^2}{2}\right)$. This function is nonincreasing on the first variable in the interval $[0,a_0]$ and nondecreasing in $[a_0 , \pi ]$, since
\[
E'(r,d)=-\frac{h(r)}{\sin r}(w'_d(r))^2  .
\]
As a consequence of this fact we can easily obtain the following result:
\begin{lemma}\label{Lemma:no zeroes}
If $0< d\leq 1$, then $w_d (r)>0 $  for all  $r\in [0,a_0]$
\end{lemma}

\begin{proof} Let $d\in(0,1]$ and suppose, to get a contradiction, that $w_d$ has a zero $r_0$ in $(0,a_0]$. Notice that $w'_d(r_0)\neq 0$, otherwise $w_d\equiv 0$ by uniqueness of the solution. Observe that $E(0,d)=\frac{n(n-2)}{4 \ell^2} (\frac{d^{p_n+1}}{p_n+1} -\frac{d^2}{2} ) <0$ in this situation. As the energy $E(r,d)$ is non increasing on the first variable, we have that $0>E(0,d)\geq E(r_0,d)=\frac{(w_d '(r_0))^2}{2}>0$, a contradiction.
\end{proof}

Note, in general, that if $d\neq 0$ and $w_d (r_1 ) =0$ then  $E(r_1 ,d)>0$. Also if $w_d ' (r_2) =0$ and
$w_d (r_2 ) \in (-1,0) \cup (0,1)$, then $E(r_2 ,d) <0$.  So, for instance, it cannot happen that
$r_2 < r_1 \leq a_0$ or $a_0 \leq r_1 <r_2$.

\vspace{.4cm}

Now consider the curve $I: \re \rightarrow \re^2$ given by $I(d) = (w_d (a_0 ) , w_d ' (a_0 ) )$. Note that $I(1)=(1,0)$,
$I(0) = (0,0)$, $I(-d) =-I(d)$ and $I(d) \neq (0,0)$ if $d\neq 0$. It is then easy to see that we have
a well defined continuous function $\theta : (0, \infty ) \rightarrow \re$ such that $\theta (1)=0$
and $\theta (d)$ gives an angle between $I(d)$ and the positive $x$-axis for any $d>0$. Note that, in a similar way, there is a unique continuous function  $\theta : (-\infty ,0) \rightarrow \re$ such that $\theta (-1) = -\pi$
and $\theta (d)$ gives an angle between $I(d)$ and the positive $x$-axis. Thus, we have that for any
$d>0$, $\theta (-d) = \theta (d) - \pi $. Also notice that $w_d (a_0 )=0$ if and only if $\theta (d) = -\frac{\pi}{2} - k \pi$ for some  integer $k$.

For $d\neq 0$ define $n(d)$ as the number of zeroes of $w_d$ in $[0, a_0 )$. Note that $n(d)=n(-d)$. We will see
that $\theta (d)$ determines $n(d)$. To prove this we start with the following observation:

\begin{lemma}\label{Lemma:Angle transversality}
Suppose $\theta(d_\ast)=-\frac{\pi}{2}-k\pi$ for some $d_\ast>0$ and some integer $k\geq 0$, and that $n(d_\ast ) =m \geq 0$. Then, given $0<\varepsilon<\pi$ there exists $\delta>0$ such that if $\vert d - d_\ast\vert<\delta$, then $\theta(d)\in (-\frac{\pi}{2}-k\pi-\varepsilon,-\frac{\pi}{2}-k\pi+\varepsilon)$ and
\begin{enumerate}
\item  $\theta(d)<-\frac{\pi}{2}-k\pi$ iff $n(d)=m+1$,
\item  $\theta(d)\geq-\frac{\pi}{2}-k\pi$ iff $n(d)= m$.
\end{enumerate}
\end{lemma}

\begin{proof}First choose $\delta_1 >0$ such that if $\vert d - d_\ast\vert<\delta_1$, then $\theta(d)\in (-\frac{\pi}{2}-k\pi-\varepsilon,-\frac{\pi}{2}-k\pi+\varepsilon)$.

If $k$ is even, then $w_{d_\ast}' (a_0 ) <0$. For $\epsilon_1 >0$ small enough, $w_{d_\ast}' (t ) <0$ for
all $t\in (a_0 - \epsilon_1 , a_0 + \epsilon_1 )$. Then there exists a positive $\delta_2 < \delta_1 $ such that
if $\vert d - d_\ast\vert<\delta_2 $, then $w_{d}' (t ) <0$ for
all $t\in (a_0 - \epsilon_1 , a_0 + \epsilon_1 )$. We can also assume that $w_d$ has exactly one zero
in $(a_0 - \epsilon_1 , a_0 + \epsilon_1 )$. There exists also $\delta < \delta_2$ such that
if  $d \in (d_\ast -\delta , d_\ast + \delta )$, then $w_d$ has exactly $m$ zeroes in $[0, a_0 - \epsilon_1 ]$.
For $d \in (d_\ast -\delta , d_\ast + \delta )$ we have that
if  $\theta(d)\geq -\frac{\pi}{2}-k\pi$  then $w_d (a_0 ) \geq 0$. This implies that the zero of $w_d$ in
$(a_0 - \epsilon_1 , a_0 + \epsilon_1 )$
is
$ \geq a_0$. Therefore $n(d)=m$. If instead $\theta(d) <  -\frac{\pi}{2}-k\pi$ then $w_d (a_0 ) < 0$. This implies that the zero of $w_d$ in
$(a_0 - \epsilon_1 , a_0 + \epsilon_1 )$
is
$ < a_0$. Therefore $n(d)=m +1$.

The argument in the case $k$ is odd is similar.
\end{proof}

As usual for $x \in \re$, let $[x]$ be the maximum integer such that $[x]\leq x$. Then for $d>0$ define

$$\overline{n} (d) = -\left[\frac{\theta(d) - \pi /2}{\pi} \right] -1.$$

\begin{proposition}\label{Lemma:argument formula}
For $d>0$ we have that $\theta (d) < \pi/2$ and $n(d) = \overline{n} (d)$.
\end{proposition}

\begin{proof} Let $A=\{ d>0 : \theta (d) < \pi /2, \ n(d) = \overline{n} (d) \}$. Note that Lemma \ref{Lemma:no zeroes} implies that $\theta (d) \in (-\pi /2 ,\pi /2 )$ for any $d \in (0,1]$. Then $(0,1] \subset A$.
If $d_\ast  \in A$ and
$\theta(d_\ast) \neq -\frac{\pi}{2}-k\pi$ for any integer $k\geq 0$ then $w_{d_\ast} (a_0 ) \neq 0$. Then there exist
$\delta >0$ such that for all $d\in (d_\ast -\delta , d_\ast + \delta )$, $w_d (a_0 )$ has the same sign as  $w_{d_\ast} (a_0 ) \neq 0$.
It follows that for all such $d$, $n(d) = n(d_\ast ) = \overline{n}(d_\ast ) = \overline{n}(d)$, and $\theta (d) < \pi /2$. So $(d_\ast -\delta , d_\ast + \delta )
\subset A \}$. If $d_\ast  \in A$ and
$\theta(d_\ast) = -\frac{\pi}{2}-k\pi$ for some integer $k\geq 0$ then  $n(d_\ast ) =k$ and by Lemma \ref{Lemma:Angle transversality} there exist
$\delta >0$ such that $(d_\ast -\delta , d_\ast + \delta )
\subset A$. Therefore $A$ is open. Assume that $A \neq (0,\infty )$. Let $d_\ast = \inf ( \  (0,\infty ) -A  \ )$.
If $\theta (d_\ast ) = \pi /2$ then $w_{d_\ast} (a_0 ) =0$ and $w_{d_\ast}' (a_0 ) >0$. It follows that
$w_{d_\ast}$ must have a zero in $(0,a_0 )$. This implies that $w_{d_\ast}$ cannot be approximated by
functions which are positive on $[0, a_0 ]$. But from the definition of $d_\ast$, for $d< d_\ast$, $\theta (d) < \pi /2$.
If such $d$ is close to $d_\ast$, then $\theta (d)$ is close to $\pi /2$ and so $w_d$ is positive in $[0, a_0 ]$, because $d \in A$.
This is a contradiction and therefore $\theta (d_\ast ) < \pi /2$.
If $\theta(d_\ast) \neq -\frac{\pi}{2}-k\pi$ for any integer $k\geq 0$ then both $n$ and $ \overline{n}$ are constant close
to $d_\ast$. Since $d_\ast$ is in the closure of $A$, we would have that $d_\ast \in A$, which is a contradiction.
If
$\theta(d_\ast) = -\frac{\pi}{2}-k\pi$ we can use again  Lemma \ref{Lemma:Angle transversality} and the fact that for any $d< d_\ast$, $d\in A$ to see
that necessarily $n(d_\ast ) =k$ and, therefore, $d_\ast \in A$, giving again a contradiction.
Therefore $A=(0,\infty )$ and the proposition is proved.

\end{proof}

 An immediate consequence of this Proposition  and Theorem \ref{Prop:prescribed zeroes} is the following
\begin{corollary}\label{Cor:Limit angle}
\[
\lim_{d\rightarrow\infty}\theta(d)=-\infty.
\]
\end{corollary}

\smallskip

Now we proceed to define a second curve in the phase space, corresponding to the solutions to problem \eqref{Eq:Singular BVP sphere radial} in $[0,\pi]$ with condition $w'(\pi)=0$. Let $\widetilde{h}(r)=-h(\pi-r)=\frac{m_1+m_2}{2}\cos r + \frac{m_2-m_1}{2}$ and consider the initial conditions problem \eqref{Eq:Singular backward equivalent}. As it was mention in Section \ref{Sec:The reduce equation}, if $\omega$ is a solution to this problem, then $\widetilde{w}(r):=\omega(\pi-r)$ solves the ``final'' conditions problem \eqref{Eq:Singular backward}.

For $c\in\R$, denote by $\widetilde{w}_c$ the solution to the problem \eqref{Eq:Singular backward} and define the map $J(c):=(\widetilde{w}_c(a_0),\widetilde{w}_c'(a_0))$.

In an entirely similar way, $J(1)=(1,0)$, $J(0)=(0,0)$, $J(c) \neq (0,0)$ if $c\neq 0$ and $J(-c)=-J(c)$. So, there is a well define argument function $\vartheta$ such that
\[
J(d)=(\vert J(c)\vert\cos(\vartheta(c)),\vert J(c)\vert\sin(\vartheta(c)).
\]

Note that if $\theta$ is the argument function corresponding to the solutions of equation
\eqref{Eq:Singular backward equivalent}, then $\vartheta =- \theta$.

It follows from the previous discussion that  $\vartheta(c)>-\pi/2$ for every $c\neq 0$ and that if $N(c)$ denotes the number of zeroes of $\widetilde{w}_c$ in $(a_0 , \pi)$, then
\[
N(c)= - \left[ \frac{-\vartheta(c) - \pi /2}{\pi}\right]  -1 .
\]
Hence, Theorem \ref{Prop:prescribed zeroes} also implies
\begin{equation}\label{Eq:limit argument}
\lim_{c\rightarrow\infty}\vartheta(c)=\infty
\end{equation}

Observe that an easy consequence of the formulas for $n(d)$ and $N(c)$ is that if $\theta(d),\vartheta(c)\in (-\pi/2,\pi/2)$, then the solutions $w_d$ and $\widetilde{w}_c$ have no zeroes before $a_0$ and after $a_0$, respectively.

\vspace{.4cm}

Next, define curves $R,S:[1,\infty)\rightarrow\R\times\R_{>0}$ in the radius-argument plane given by
\[
R(d):=(\theta(d),\vert I(d) \vert)\qquad\text{and}\qquad S(c):=(\vartheta(c),\vert J(c)\vert)
\]
In the following lemma we use strongly that the positive solutions to the Yamabe problem on the sphere are all axially symmetric, i.e. invariant by an action of
$O(n)$ fixing an axis.

\begin{lemma}
The curves $R$ and $S$ are simple and they intersect only at the point $(0,1)$.
\end{lemma}

\begin{proof}
The fact that $R$ and $S$ are simple follows immediately from the uniqueness of the solutions to the problems \eqref{Eq:Singular forward} and \eqref{Eq:Singular backward}.  If $d,c>1$ are such that
$\theta (d) = \vartheta (c)$ then, by Proposition \ref{Lemma:argument formula}, $\theta(d) = \vartheta(c)\in(-\pi/2,\pi/2).$ Hence $w_d = \widetilde{w}_c$
is a solution to the problem \eqref{Eq:Singular BVP sphere radial} without zeroes, yielding a positive solution to the Yamabe problem which
is not axially symmetric. But positive solutions of the Yamabe problem on the sphere are all axially symmetric \cite{a1}, giving a contradiction to the hypothesis $m_1,m_2<n-1$. Therefore, $R\cap S=\{(0,1)\}$.
\end{proof}

As a consequence of this lemma, the curve $T:=R\cup S$ defines a simple curve in $\R^2$. As $\theta(d)\rightarrow-\infty$ and $\vartheta(c)\rightarrow\infty$, the Jordan curve Theorem on the sphere implies that the curve $T$ divides $\R\times\R_{>0}$ into two open, simply connected and disjoint components. Let $\mathcal{D}$ be the component such that the $x$-axis is contained in its closure and let $\mathcal{U}$ be the other one. For each $i,j\in\mathbb{N}\cup\{0\}$ define the exit times
\[
d_{i}:=\max\{d\geq 1\;:\;\theta(d)=-i\pi\}\qquad \text{and}\qquad c_{j}:=\max\{c\geq 1\;:\;\vartheta(c)=j\pi\}
\]
which are well defined by virtue of Corollary \ref{Cor:Limit angle} and limit \eqref{Eq:limit argument}.

For each $i,j$, define sequences $(x_i)$ and $(y_j)$ by
\[
x_i:=\vert I(d_i) \vert\qquad\text{and}\qquad y_j:=\vert J(c_j )\vert.
\]

\begin{lemma}\label{Lemma:sequences}
The sequences $(x_i)$ and $(y_j)$ are either monotone increasing or monotone decreasing and $ x_0,y_0\leq 1$
\end{lemma}

\begin{proof}
We will show the lemma for $(x_i)$, the proof for $(y_j)$ being entirely analogous. First we show that $x_0\leq 1$. Notice that if $x_0 >1$, then the solution $w_{d_0 }$ has a local maximum at $a_0$.  Since $w_{d_0}$ also has a local maximum at $0$ then it must have a local minimum in $(0,a_0 )$. Let $d_*$ be
the first value after $d_0$ such that $\theta (d_* )= - \frac{\pi}{2}$.  Since $w_{d_* } (a_0 )=0$ we know from
the comments after Lemma 4.1 that $w_{d_*}$ cannot have a local minimum in $(0, a_0 )$ and
so $w_{d_*}' <0 $ in $(0, a_0 )$. Therefore if $d$ is close to $d_*$, then $w_d$ does not have
a local maximum in $(0,a_0 )$. It follows from the definition of
$d_0$ and $d_*$ that if $d \in (d_0 , d_* )$ then $w_d (a_0 ) >0$ and $w_d ' (a_0 ) <0$. This implies in
particular that if $d$ is close to $d_0$, $d >d_0$, then $w_d$ has a local maximum in $(0,a_0 )$. It also 
implies that if $d\in (d_0 , d_* )$ and $w_d$ does not have a local maximum in $(0, a_0 )$ then $w_d ' <0$
in $(0,a_0 )$. But this yields that $\{ d \in (d_0 , d_* ) : w_d$ does not have a local maximum in $(0,a_0 ) \}$ is
open. It is also clear that  $\{ d \in (d_0 , d_* ) : w_d$  has a local maximum in $(0,a_0 ) \}$ is open. From the previous comments both are non-empty and give a partition of $(d_0 , d_* )$. This yields a contradiction
and therefore $x_0 \leq 1$.

Now, to prove the monotonicity of the sequence, for each $i\in\N$ $i\neq 0$, define the curves
\[
R-(i\pi,0):=\{R(d)-(i\pi,0)\;:\;d\geq 1\},
\]
which is just the curve $R$ displaced to the left by $i\pi$. We claim that $R-(i\pi,0)\cap T=\emptyset$ for every $i\geq 1$. Indeed, as $\vartheta>-\pi/2$, the curves $R-(i\pi,0)$ and $S$ never intersect. Now, if $R$ and $R-(i\pi,0)$ intersect, then we contradict the fact that the curve $I$ is simple.  Hence, the curves $R-(i\pi,0)$ do not intersect $T$ and they are contained either in $\mathcal{D}$ or in $\mathcal{U}$. As $d_0$ is the exit time of $R$ from the line $\theta=0$, then $d_0$ is the exit time of $R-(\pi,0)$ from the line $\theta=-\pi$. Consider the points $R(d_0)=(0,x_0)$, $R(d_1)=(-\pi,x_1)$ and $R(d_0)-(\pi,0)=(-\pi,x_0)$. Observe $x_0<x_1$ if and only if $R-(\pi,0) \subset \mathcal{D}$, since the curve $R$ restricted to the interval $[d_0,d_1]$ is homotopic in $\R\times\R_{>0}\smallsetminus\{(-\pi,x_0),(-\pi,x_1)\}$ to the straight segment joining $(0,x_0)$ and $(-\pi,x_1)$. As $T$ divides $\R\times\R_{>0}$ only in two disjoint open sets and as $x_0\neq x_1$, this also implies $x_1<x_0$ if and only if $R-(\pi,0)\subset\mathcal{U}$. Suppose now that $R-(\pi,0)$ is under $R$, i. e.
$R-(\pi,0)\subset \mathcal{D}$, so that $x_0<x_1$. Now for any $i\geq 1$ we look at the exit times of $R-(\pi,0)$ and $R$ at
the line $\theta =-i\pi$; since $R-(\pi,0)\subset \mathcal{D}$ it follows that $x_i  > x_{i-1}$ and the sequence is increasing.

In case $x_0 > x_1$, in a similar way one proves that the sequence is decreasing.

\end{proof}

One can prove  Theorem \eqref{Th:Nodal Yamabe ODE}  in case either of the sequences is decreasing using
 the next result.

\begin{lemma}\label{Lemma:sequences decreasing}
If the sequence  $ (x_i) $ is decreasing, then for any integer $k\geq 1$ there exists $\alpha_k \in (d_k , d_{k+1})$
such that $w_{\alpha_k}'(0)=0=w_{\alpha_k}'(\pi)$ and $w_{\alpha_k}$ has exactly k zeroes in $[0,\pi]$ which are all in $(0,a_0 )$. A similar statement  holds true when the sequence $(y_i)$ is decreasing, but now with the zeroes lying in $(a_0,\pi)$.
\end{lemma}

\begin{proof} By Lemma \ref{Lemma:sequences}, we have that $x_i  <1 $ for all $i\geq 1$. Now, for $k\geq 1$ fixed, the solution $w_{d_k}$ must have exactly $k$ zeroes in
$(0,a_0 )$ and a local extremum in $a_0$ with value $x_k$ or $-x_k$. It follows that there must be another local extremum
between the last zero and $a_0$ (for instance if $w_{d_k} (a_0 ) >0$, then $a_0$ is a local minimum of
$w_{d_k}$ and therefore there must be a local maximum between the last zero and $a_0$).

Let $d_*$ be the first time after $d_k$ such that the curve $I$ touches the $y$-axis.
Note that $\theta (d_* ) = \theta  (d_k ) -\pi /2  =-k\pi - \pi /2$.
The solution $w_{d_*}$ has exactly
$k$ zeroes in $(0,a_0 )$ and another one at $a_0$. Recall that the energy is decreasing in the interval $[0,a_0 ]$ and increasing in the interval $[a_0 , \pi ]$. Since $E(a_0 , d_* ) >0$, we have that 
$E(r, d_* ) >0$ for all $r\in [0 , \pi ]$. We define the  set $A \subset [d_k , d_* ]$  by saying that $d \in A$ if and only if  $w_d$
has at least $k$ zeroes in $(0,a_0 )$ and after the $k$-th zero (and before the $(k+1)$-th zero) it has a local extremum 
$r_d$ with value in $(-1,1)$. This again implies that $w_d$ must have another extremum before $r_d$ and after the
$k$-th zero. Observe that $d_k \in A$ and that if $d$ is in the closure of $A$ then $w_d$ must have at least $k$ zeroes
in $(0, a_0]$. Notice that $A$ is open in $[d_k,d_*]$ and that $d_* \notin A$: since $w_{d_*} (a_0 ) =0$ the energy must be positive 
in $[0, \pi ]$ and therefore $w_{d_*}$ cannot have an extremum with value in $(-1,1)$. But for $d$ close to 
$d_*$ we also have $E(a_0 ,d )>0$ and therefore $d\notin A$. Let $[d_k , \alpha_k )$ be the connected component
of $A$ that contains $d_k$. Note that $\alpha_k \in (d_k , d_* )$. Since $\alpha_k$ is in the closure of $A$ we have that
$w_{\alpha_k}$ must have at least $k$-zeroes in $(0,a_0 ]$. Let $\widehat{r}_k$ be the $k$-th zero of $w_{\alpha_k}$. Assume
that $w_{\alpha_k}$ has a $(k+1)$-th zero $\widehat{r}_{k+1}$. Then for $d$ close to $\alpha_k$ the solution
$w_d$ must also have at least $k+1$ zeroes in $(0,\pi)$. If $d<\alpha_k$ then $d\in A$. So $w_d$ has an extremum with
value in $(-1,1)$ between its $k$-th zero and its $(k+1)$-th zero. Let $r_d$ be this extremum of $w_d$. Pick a
sequence $d_i \rightarrow \alpha_k$ such that the sequence $r_{d_i}$ converges. We must have 
$r_0= \lim_{i\rightarrow\infty} r_{d_i}  \in [\widehat{r}_k , \widehat{r}_{k+1} ]$, $w_{\alpha_k } (r_0 ) \in [-1,1]$, $w_{\alpha_k} ' (r_0 )=0$. This, the fact that $\alpha_k\neq -1,0,1$ and uniqueness of the solutions imply that $r_0= \lim_{i\rightarrow\infty} r_{d_i}  \in (\widehat{r}_k , \widehat{r}_{k+1} )$ and $w_{\alpha_k } (r_0 ) \in (-1,1)$. This would mean that
$\alpha_k \in A$ contradicting that $[d_k , \alpha_k )$ was a connected component of the open set $A$. Therefore
$w_{\alpha_k}$ has exactly $k$ zeroes in $(0,\pi )$. 

Assume that  $w_{\alpha_k}$ is monotone and unbounded after $\widehat{r}_k$. Then the energy blows to infinity after
$a_0$. Thus for $d$ close to $\alpha_k$, $d<\alpha_k$, we have that there exists $r>a_0$ such that after  its $k$-th zero, $w_d$ is
monotone until $r$ and $E(r,d) >0$. Since the energy is increasing after $r$ this implies, as before, that $w_d$ cannot have
an extremum with value in $(-1,1)$ after its $k$-th zero. Therefore $d\notin A$, which contradicts the definition of
$\alpha_k$. Thus, $w_{\alpha_k}$ is monotone and bounded after $\widehat{r}_k$ or it has an extremum after $\widehat{r}_k$. 
Assume the second case, so $w_{\alpha_k}$ has a local extremum at some point $r_e \in (\widehat{r}_k , \pi )$. We can see
exactly as before that if $w_{\alpha_k}$ does not have any critical point after $r_e$ then it must be monotone and bounded in $(r_e , \pi )$. If $w_{\alpha_k}$ has at least two local extrema after $\widehat{r}_k$, $r_e , r_f$ then since
$\alpha_k \notin A$ we must have that $w_{\alpha_k} (r_e ), w_{\alpha_k} (r_f ) \notin (-1,1)$. But this implies that
$w_{\alpha_k} (r_e )$ and  $w_{\alpha_k} (r_f )$ have different signs, implying that $w_{\alpha_k}$ has a
$(k+1)$-th zero. We have already proved that this is not the case. So we proved that there exists $r>a_0$ such
that $w_{\alpha_k}$ is monotone, bounded and does not have a zero in $(r, \pi )$. Moreover, if $w_{\alpha_k}'$ is
not bounded in $(r,\pi )$ then the energy becomes positive and exactly as before we would prove that for
$d$ close to $\alpha_k$, $d<\alpha_k$, $d \notin A$. Therefore $w_{\alpha_k}'$ is also bounded in $(r, \pi )$. This implies that
$w_{\alpha_k}$ is a $C^2$-function on $[0,\pi ]$ and $w_{\alpha_k} ' (\pi )=0$. This is elementary but the details are
lengthy, so we will write the details in Appendix A at the end of the article. In Lema A.1 we will prove that
$w_{\alpha_k}$ is actually $C^1$ and $w_{\alpha_k}' (\pi )=0$. And in Lemma A.2 we will prove that
$w_{\alpha_k}$ is actually $C^2$ in $[0,\pi ]$, giving a bona fide solution with
exactly $k$ zeroes. This proves the lemma when $(x_i )$ is decreasing and the proof when $(y_i )$ is 
decreasing is entirely similar.  
\end{proof}

We have now everything to prove Theorem \ref{Th:Nodal Yamabe ODE}

\smallskip

\begin{proof}[\textbf{Proof of Theorem \ref{Th:Nodal Yamabe ODE}.}]
By the preceding lemma, in case  either $(x_i)$ of $(y_j)$ is decreasing Theorem \ref{Th:Nodal Yamabe ODE} holds. Therefore
we can assume that both sequences $(x_i)$ and $(y_j)$ are increasing.
We consider the simple  curve $T =R \cup S$ that divides the open half space in the two open
connected subsets $\mathcal{U}$ and $\mathcal{D}$.

By the hypothesis (see the proof of Lemma \ref{Lemma:sequences decreasing}), for any positive integer $k$ the curve $R-k(\pi ,0) \subset \mathcal{D}$, and $S -k(\pi ,0)$ must intersect $\mathcal{U}$ since it is over $S$. Then $T-k(\pi ,0)$ must intersect $T$. This implies that $R$ must intersect $S -k(\pi ,0)$. Let $d_R >1$ and
$c_S >1$ be the points such that $R(d_R ) = S(c_S ) -(k\pi ,0)$.

Consider the homotopy equivalence given by the projection $p:\R^2 - \{ (0,0) \} = \R \times \S^1 \rightarrow \S^1$. The
curve $I$ restricted to $[1, d_R ]$ projects to a curve in $\S^1$ homotopic with fixed endpoints to the curve
$c:[0, -\theta ( d_R )] \rightarrow S^1$, $c_1 (t) = (  \cos (-t) , \sin -(t)  )  $. Note that the number of zeroes of $w_{d_R}$
in $[0,a_0 ]$ is equal to  the number of times that the curve $c_1$ intersects the $y$-axis. Similarly the projection of
the curve $J$ restricted to $[1,c_S]$ is homotopic to the curve $c_2 : [0, \vartheta (c_S )]\rightarrow\S^1$, $c_2 (t) =
(\cos t , \sin t )$. And the number of zeroes of $ \widetilde{w}_{c_S}$ in $[a_0 ,\pi ]$ is equal to the number of times $c_2$ 
intersects the $y$-axis.

If $k$ is even then $I(d_R ) = J(c_S )$. In this case, the restriction of $J$ to $[1,c_S ]$ followed by the restriction of
$I$ to $[1,d_R ]$ (in the opposite direction) form a closed curve $C$.  Since $R(d_R ) = S(c_S ) - (k\pi ,0)$ we
have that $\vartheta (c_S ) - \theta (d_R ) = k \pi$.  Then the projection of $C$ to $\S^1$ is homotopic to
$c: [0,k \pi ] \rightarrow \S^1$, $c(t) = (\cos t , \sin t )$. And it follows from the previous comments that 
$w_{d_r} =\widetilde{w}_{c_S}$ is a solution with exactly $k$ zeroes in $[0,\pi ]$.

If $k$ is odd then $I(d_R ) = -J(c_S )$. Now the restriction of $J$ to $[1,c_S ]$ followed by the restriction of
$-I$ to $[1,d_R ]$ (in the opposite direction) form a curve $\widetilde{C}$ from $(1,0)$ to  $(-1,0)$. The
projection of $\widetilde{C}$ to $\S^1$ is homotopic, with fixed endpoints to the curve $c:[0,k \pi ]\rightarrow\S^1$,
$c(t) = (\cos t , \sin t )$. And it now follows from the previous comments that $w_{d_R} = - \widetilde{w}_{c_S}$
is a solution with exactly $k$ zeroes in $[0,\pi ]$.
\end{proof}

\appendix

\section{}

In this appendix we give the details of the last argument in the proof of Lemma 4.7.

\begin{lemma}  Let $P$ be a continuous odd function, $\epsilon >0$,  and $u\in C^2 [\pi - \epsilon , \pi )$ be a solution of

\begin{equation}\label{A1}
u'' + \frac{h(r)}{\pi -r} u' + P(u) =0  \ \ \ \ \ \ \ \ in \ \ \ \ [\pi -\epsilon , \pi )
\end{equation}

\noindent
Assume that
\begin{itemize}
	
	\item[i)] $u$ is monotone and bounded in $[\pi -\epsilon , \pi )$
	
	\item[ii)]$ h<0 $  in $[\pi -\epsilon , \pi ]$ and is continuous
	
	\item[iii)] $u'$ is bounded in $[\pi -\epsilon , \pi )$
	
\end{itemize}

\noindent
Then $u\in C^1 [\pi - \epsilon , \pi ]$ and $u'(\pi )=0$.
\end{lemma}

\begin{proof} Assume that $u$ is monotone increasing (in case it is decreasing one could simply consider
	$-u$, which is also a solution since $P$ is odd). We have $u' \geq 0$ and assume that $\lim_{t\rightarrow \pi } u'(t) \neq 0$. Then there exists 
	$\delta >0$ and a sequence $t_k \rightarrow \pi$ such that $u' (t_k ) > \delta$.
	For $k$ large this implies that $u''(t_k )>0$. Then taking $\epsilon$ smaller if necessary we can
	assume that $u'(t) >\delta$  $\forall t \in [\pi - \epsilon , \pi )$. 
	And therefore $u'$ is increasing in $[\pi -\epsilon , \pi )$.
	
	Choosing again $\epsilon$ smaller if necessary and letting $\lambda =-h(\pi )/2 >0$, we can assume that $u''=-\frac{h(t)}{\pi - t} u' -P(u)$
	$>\frac{\lambda}{\pi -t} u'$.  Then we have that

	$$\frac{u''}{u'} > \frac{\lambda}{\pi -t} .$$

	Then 
	
	$$(ln (u' ))' > (-\lambda ln (\pi -t))'$$
	
	\noindent
	and we can deduce that for some constant $c \in \R$ 
	
	$$ln(u') > -\lambda ln (\pi -t ) + c .$$
	
	Hence, 
	
	$$u' > e^c (\pi -t)^{-\lambda}.$$
	
	This implies that $u'$ is not bounded which is a contradiction that came from assuming that
	$\lim_{t\rightarrow \pi } u'(t) \neq 0$. Therefore $\lim_{t\rightarrow \pi } u'(t) = 0$.
	
	Moreover,
	
	$$\lim_{t\rightarrow 0} \frac{u(\pi -t ) -u(\pi ) }{-t} =\lim_{t\rightarrow 0} \frac{u(\pi ) -u(\pi -t ) }{t} = \lim_{t\rightarrow 0} \left( \lim_{\delta \rightarrow 0}\frac{u(\pi -\delta ) -u(\pi  -t) }{t-\delta}  \right)$$
	
	But
	
	$$\frac{u(\pi -\delta ) -u(\pi  -t) }{t-\delta}  = u'(\beta )$$ 
	
	\noindent
	for some $\beta \in (\pi -t  , \pi -\delta )$, which implies that the limit is 0 and therefore $u\in C^1 [\pi -\epsilon , \pi]$.

\end{proof}

\begin{lemma} Let $r>0$ and $u\in C^1 (\pi -r , \pi] \cap C^2 (\pi -r , \pi )$ be a solution of equation (\ref{A1}) such that $u'( \pi )=0$. Then $u\in C^2 (\pi - r , \pi ]$ and it solves equation (\ref{A1}) on $(\pi - r , \pi ]$.
\end{lemma}

\begin{proof} Note that if $u \in C^2 (0, \pi ]$  then we would have 
	$u''(\pi ) =\lim_{r\rightarrow 0} \frac{-u'(\pi -r)}{r} = \frac{-P(u(\pi ))}{1- h(\pi )}$. Let $C= \frac{-P(u(\pi ))}{1-h(\pi )}$.

	Assume then that 
	
	$$\lim_{r \rightarrow 0} \frac{-u'(\pi - r)}{r} =C$$

	Then of course $u''( \pi ) =C$. Also

	$$u''( \pi - r) = -\frac{h(\pi - r)}{r} u' (\pi -r )-P(u (\pi - r) ) \rightarrow h(\pi ) C -P(u(\pi )) =C,$$
	
	\noindent
	and we have that $u\in C^2 (\pi - r,\pi ]$ and it solves equation (\ref{A1}) on $(\pi -r ,\pi ]$.
	
	Therefore we only need to prove 
	
	$$\lim_{r \rightarrow 0} \frac{- u'(\pi - r)}{r} =C$$

	Assume it is not the case. Assume for instance that $\liminf_{r\rightarrow 0}  \frac{-u'(\pi - r)}{r}  < C$
	(the case $\limsup_{r\rightarrow 0}  \frac{-u'(\pi - r)}{r}  > C$ is similar).

	Then there  exists $\delta >0$ such that for any $r_0 >0$ there exists $r_1 \in (0, r_0 )$ such that 
	
	$$ \frac{-u'(\pi - r_1 )}{r_1} <C -\delta . $$
	
	\noindent
	For any $r_1$ verifying the previous inequality we have
	
	$$u''(\pi - r_1 ) = -P(u(\pi - r_1 )) - h(\pi - r_1 ) \frac{u'(\pi - r_1 )}{r_1} > -P(u(\pi - r_1 )) + h(\pi - r_1 ) (C-\delta ).$$
	
	\noindent
	But $-P(u(\pi - r_1 )) + h(\pi - r_1 ) C \rightarrow C$. Therefore by taking $r_1 $ small enough we can assume that 
	$u''( \pi - r_1 )>C$.

	Since
	
	$$\left( \frac{-u'(\pi - r)}{r} \right) ' =\frac{u''(\pi - r) r + u'(\pi - r)}{r^2} = \frac{1}{r} (u'' (\pi - r) + u' (\pi - r )/r)$$
	
	\noindent
	we have that $(-u' (\pi -r)/r)' (r_1 ) >0$. But this implies that $-u' (\pi -r )/r$ is $<C -\delta $ and 
	increasing in $(\pi -r_1 , \pi )$.
	Let $D = \lim_{r\rightarrow 0} -u' (\pi -r ) /r \leq  C-\delta$. Then
	
	$$\lim_{r\rightarrow 0} u''(\pi - r) = h(\pi )D -P(u(\pi )) \geq  h( \pi )C - h(\pi ) \delta -P(u(\pi ))= C - h(\pi ) \delta >C .$$

	We have that for $r<r_1$
	
	$$\lim_{\alpha \rightarrow 0} \frac{-u'(\pi - r) +u'(\pi - \alpha )}{r-\alpha }= \frac{-u'(\pi - r)}{r} < C -\delta .$$
	
	But also by the mean value theorem there exists $r_{\alpha} \in (\alpha , r)$ such that 
	
	$$\frac{-u'(\pi - r)  + u'(\pi - \alpha )}{r-\alpha } = u''(r_{\alpha }) >C$$
	
	This is a contradiction which came from assuming that $\liminf_{r\rightarrow 0}  \frac{-u'(\pi - r)}{r}  < C$.
	Therefore $\liminf_{r\rightarrow 0}  \frac{-u'(\pi - r)}{r}  \geq  C$ and 
	similarly $\limsup_{r\rightarrow 0}  \frac{-u'(\pi - r)}{r}  \leq C$.
	Therefore $\lim_{r\rightarrow 0}  \frac{-u'(\pi - r)}{r}  = C$ and the lemma is proved.
	
\end{proof}

\end{document}